\documentclass{article}

\usepackage{amsmath,amsfonts,amssymb}

\usepackage{graphicx}

\usepackage{dsfont}
\usepackage{bbold}
\usepackage{amsthm}
\usepackage[utf8]{inputenc}
\usepackage{appendix}
\usepackage{geometry}

\newtheorem{thm}{Theorem}

\newtheorem{prp}{Proposition}
\newtheorem{claim}{Claim}
\newtheorem{lemma}{Lemma}

\newtheorem{conj}{Conjecture}

\title{\textbf{Restricting directions for Kakeya sets}}

\author{Anthony Gauvan}

\begin{document}

\maketitle

\begin{abstract}

We prove that the \textit{Kakeya maximal conjecture} is equivalent to the \textit{$\Omega$-Kakeya maximal conjecture}. This completes a recent result in \cite{KELETI} where Keleti and Mathé proved that the \textit{Kakeya conjecture} is equivalent to the \textit{$\Omega$-Kakeya conjecture}. Moreover, we improve concrete bound on the Hausdorff dimension of a $\Omega$-Kakeya set : for any Bore set $\Omega$ in $\mathbb{S}^{n-1}$, we prove that if $X \subset \mathbb{R}^n$ contains for any $e \in \Omega$ a unit segment oriented along $e$ then we have $$d_X \geq \frac{6}{11}d_\Omega +1$$ where $d_E$ denotes the Hausdorff dimension of a set $E$.
\end{abstract}

\section{Introduction}

The \textit{Kakeya} problem is a central question in harmonic analysis which can be formulated in different ways ; it is also related to \textit{restriction} theory and arithmetic. A measurable set $X$ in $\mathbb{R}^n$ is said to be a \textit{Kakeya} set if for any direction $e \in \mathbb{S}^{n-1}$ it contains a unit segment $T_e$ oriented along $e$. The Kakeya conjecture concerns the Hausdorff dimension of Kakeya set $X$.

\begin{conj}[Kakeya conjecture]\label{CONJ1} If $X$ is a Kakeya set in $\mathbb{R}^n$ then $$d_X = d_{\mathbb{S}^{n-1}} + 1 = n.$$
\end{conj}

This conjecture has been proved by Davies in the plane in \cite{DAVIES}. For $n \geq 3$, a vast amount of techniques have been developed in order to tackle this issue ; we invite the reader to look at \cite{KT} or \cite{KTL} to see the extent of the techniques that might be deployed. Here, we will simply say that, specialists are able to prove that if $X$ is a Kakeya set in $\mathbb{R}^n$ then $$d_X \geq (\frac{1}{2} + \epsilon_n) d_{ \mathbb{S}^{n-1}} +1 $$ where $\epsilon_n > 0$ is a dimensional constant. There exists a more quantitative version of the Kakeya conjecture and we need to introduce the \textit{Kakeya maximal operator} to state it. We define the \textit{Kakeya maximal function} $$K_\delta f : \mathbb{S}^{n-1} \rightarrow \mathbb{R}_+$$ at scale $\delta > 0$ of a locally integrable function $f : \mathbb{R}^n \rightarrow \mathbb{R}$ as $$ K_\delta f(e) := \sup_{a \in \mathbb{R}^n} \frac{1}{\left|T_{e,\delta}(a)\right|} \int_{ T_{e,\delta}(a) }  |f|(x) dx$$ where $e\in \mathbb{S}^{n-1}$ and $T_{e,\delta}(a)$ stands for the tube in $\mathbb{R}^n$ with center $a$, oriented along the direction $e$, of length $1$ and radius $\delta$. It appears that any quantitative information on $\|K_\delta\|_{\sigma,p}$ provides lower bound on the dimension of any Kakeya set $X$ : for any $1 < p < \infty$ and $ \beta > 0$ such that $n-\beta p>0$, if we have $$\|K_\delta \|_{\sigma,p} \lesssim_{n,p,\beta} \delta^{-\beta}$$ then the Hausdorff dimension of any Kakeya set in $\mathbb{R}^n$ is at least $n-\beta p$. In regards of this fact, the following conjecture is called the \textit{Kakeya maximal conjecture}, it is stronger than the Kakeya conjecture.

\begin{conj}[Kakeya maximal conjecture]
For any $\epsilon>0$ we have $$\|K_\delta\|_{\sigma,n} \lesssim_{n,\epsilon} \delta^{-\epsilon}.$$
\end{conj}

In this text, we are concerned with a natural generalization of the Kakeya problem. Given an arbitrary Borel set of directions $\Omega \subset \mathbb{S}^{n-1}$, we say that a set $X$ in $\mathbb{R}^n$ is a \textit{$\Omega$-Kakeya} set if for any $e \in \Omega$ there exists a unit segment $T_e$ oriented along $e$ included in $X$. What can be said about the dimension of a $\Omega$-Kakeya set ? The following conjecture seems plausible.

\begin{conj}[$\Omega$-Kakeya conjecture]
For any Borel set $\Omega$ in $\mathbb{S}^{n-1}$ ; if $X$ is a $\Omega$-Kakeya set then $$d_X \geq d_\Omega +1.$$ 
\end{conj}

At least three questions can be asked. First, if we know that the Kakeya conjecture is true, can we say something about the $\Omega$-Kakeya conjecture ? Secondly, can we state a \textit{maximal} version of the $\Omega$-Kakeya conjecture ? Lastly, if there exists a $\Omega$-Kakeya maximal conjecture, what can we said about it given the Kakeya maximal conjecture ? Very recently, Keleti and Mathé gave a positive answer to the first question in \cite{KELETI}.

\begin{thm}[Keleti-Mathé]\label{CA}
If the Kakeya conjecture is true then the $\Omega$-Kakeya conjecture is also true.
\end{thm}

The proof of this Theorem relies on fine notions concerning Hausdorff and packing dimension and we invite the reader to look at \cite{KELETI} for more details.

\section{Notations}

We will work in the euclidean space $\mathbb{R}^n$ with $n \geq 3$ endowed with the Lebesgue measure and the euclidean distance ; if $U$ is a measure set in $\mathbb{R}^n$ we denote by $\left|U \right|$ its $n$-dimensional Hausdorff measure and by $\left|U \right|_k$ its $k$-dimensional Hausdorff measure for $k < n$. Also we denote by $d_U$ its Hausdorff dimension and by $ \text{diam}\left(U \right)$ its diameter. We denote by $\sigma$ the spherical surface measure on $\mathbb{S}^{n-1}$ and $\mu$ will stand for a probability measure on $\mathbb{S}^{n-1}$ ; we will denote by $S_\mu$ its support. The surface measure $\sigma$ will usually not charge the support of $\mu$ \textit{i.e.} we will have $\sigma(S_\mu) = 0$. We will see that we need to focus on the study of $(\mu,p)$-norm of $$K_\delta : L^p(\mathbb{R}^n) \rightarrow L^p(\mathbb{S}^{n-1},\mu) $$ where $\mu$ is an arbitrary measure on $\mathbb{S}^{n-1}$ \textit{i.e.} we will be interested in estimating the following quantity $$\|K_\delta\|_{\mu,p} := \sup_{ \|f\|_p \leq 1} \left(  \int_{\mathbb{S}^{n-1}} (K_\delta f(e))^p d\mu \right)^\frac{1}{p}.$$ Hence the notation $\|K_\delta\|_{\mu,p}$ emphasizes the dependence on the probability measure $\mu$ set on the target space.

\section{Results}

We are going to formulate the appropriate \textit{maximal} version of the $\Omega$-Kakeya conjecture ; then we will prove that the Kakeya maximal conjecture implies the \textit{$\Omega$-Kakeya maximal conjecture}. In other words, \textit{we prove the maximal analog to Theorem \ref{CA}}. Our approach is close to the approach iniated by Mitsis in \cite{MIT} ; here we work in higher dimension. We will start by proving the following Proposition.

\begin{prp}\label{THMD}
Let $\mu$ be an arbitrary probability measure on $\mathbb{S}^{n-1}$
and suppose we have $ 1 < p < \infty$ and $\beta > 0$ such that $ n - \beta p > 0$. Suppose that we have $$ \|K_\delta\|_{\mu, p} \lesssim_{n,p,\beta} \delta^{-\beta}.$$ In this case, for any Borel set of direction $\Omega$ containing the support $S_\mu$ of $\mu$, the Hausdorff dimension of any $\Omega$-Kakeya set $X$ is at least $n - \beta p$.
\end{prp}

In regards of this Proposition, we will call the following Conjecture the $\Omega$-Kakeya maximal conjecture.

\begin{conj}[$\Omega$-Kakeya maximal conjecture]
Fix any probability $\mu$ defined on $\mathbb{S}^{n-1}$ satisfying for some $d \in [0, n-1]$ $$\left[ \mu \right]_d := \sup_{e \in \mathbb{S}^{n-1},r >0} \mu\left( B_{e,r} \right)r^{-d}  \leq 1.$$ Then for any $\epsilon >0$ we have  $$ \left\|K_\delta\right\|_{\mu,n} \lesssim_{n,d,\epsilon} \delta^{\frac{(d+1)}{n}-(1+\epsilon)}.$$
\end{conj}

Using Frostman's Lemma, one can easily checked that the $\Omega$-Kakeya maximal conjecture implies the $\Omega$-Kakeya conjecture. One of our main result is the following.

\begin{thm}\label{C2}
If the Kakeya maximal conjecture is true then the $\Omega$-Kakeya maximal conjecture.
\end{thm}

In particular, since in the plane $\mathbb{R}^2$ we do have $\|K_\delta\|_{\sigma,2} \lesssim_{\epsilon} \delta^{-\epsilon}$ for any $\epsilon > 0$,  this gives another proof of the $\Omega$-Kakeya conjecture in the plane ; recall that this Theorem has also been established by Mitsis in \cite{MIT}.

\begin{thm}
For any Borel set $\Omega$ in $\mathbb{S}^1$, if $X$ is a $\Omega$-Kakeya set then $$ d_X \geq d_\Omega +1.$$
\end{thm}

At this point, it is interesting to note Theorems \ref{CA} and Theorem \ref{C2} \textit{cannot provide partial result} to the $\Omega$-Kakeya conjecture. For example, say we can prove that if $X$ is a Kakeya set then we have $$d_X \geq \frac{3}{4}(n-1) + 1.$$ In this situation, \textit{we cannot use Theorems \ref{CA} and \ref{C2}} - neither their methods of proof - to show that for any $\Omega$-Kakeya set $Y$, we have $$d_Y \geq \frac{3}{4}d_\Omega  + 1.$$ Hence, in order to obtain further partial result on the $\Omega$-Kakeya conjecture, we are going to employ Bourgain's arithmetic argument in order to prove the following Theorem.

\begin{thm}\label{THM0}
For any Borel set $\Omega$ in $\mathbb{S}^{n-1}$ and any $\Omega$-Kakeya set $X$ in $\mathbb{R}^n$ we have $$d_X \geq \frac{6}{11}d_\Omega + 1.$$
\end{thm}

In \cite{LAURA}, Venieri proved that if $\Omega$ is $d$-Alfhors regular then a $\Omega$-Kakeya set $X$ has Hausdorff dimension greater than $\frac{d+2}{2} + \frac{1}{2}$. Theorem \ref{THM0} strengthen this result since it gives better estimate for large $n$ and also since we do not make assumption concerning the set of direction $\Omega$.

\section{Proof of Proposition \ref{THMD}}

We let $\mu$ be an arbitrary probability measure on $\mathbb{S}^{n-1}$ and suppose we have $ 1 < p < \infty$ and $\beta > 0$ such that $ n - \beta p > 0$. We also suppose that we have $$ \|K_\delta\|_{\mu, p} \lesssim_{n,p,\beta} \delta^{-\beta}.$$ We fix then an arbitrary Borel set of directions $\Omega$ which contains $S_\mu$ and we let $X$ included in $\mathbb{R}^n$ be a $\Omega$-Kakeya set ; we are going to prove that we have $$ d_X \geq n-\beta p.$$ Fix an arbitrary $ \alpha \in (0,n-\beta p)$. Consider a covering of $X$ by balls $B_i = B(x_i,r_i)$ such that $r_i < 1$ for any $i \in \mathcal{I}$. We will show that we have $$ \sum_{i \in \mathcal{I}} r_i^\alpha \gtrsim_\alpha 1$$ which gives $d_X \geq \alpha$. For $e \in \Omega$, let $T_e \subset X$ be a unit segment oriented along the direction $e$ ; for $k \geq 1$ we order the balls $B_i$ by their radii defining $$\mathcal{I}_k = \left\{ i \in \mathcal{I} : r_i \simeq \frac{1}{2^k} \right\}.$$ We also define $$ \Omega_k = \left\{ e \in \Omega : \left| T_e \cap \bigcup_{i \in \mathcal{I}_k} B_i \right|_1 \geq \frac{1}{2k^2} \right\}.$$ It is not difficult to show that we have $\Omega = \bigcup_{k \geq 1} \Omega_k$. We are going to fatten a little bit every segment $T_e$ in order to deal with tubes. We define for $k \geq 1$ the set $$Y'_k = \bigcup_{i \in \mathcal{I}_k} B(x_i, 2r_i).$$ For $e \in \Omega$ , by simple geometry we have the following inequality $$ \left| T_{e,2^{-k}} \cap Y'_k \right| \gtrsim \frac{1}{k^2}\left| T_{e,2^{-k}}  \right|.$$ Hence for any $e \in \Omega_k$ we have $K_{2^{-k}} \mathbb{1}_{Y'_k} (e) \gtrsim \frac{1}{k^2}$. Using our hypothesis on $K_\delta$, we obtain $$\mu\left(\Omega_k \right) \lesssim \mu \left( \left\{ K_{2^{-k}} \mathbb{1}_{Y'_k} \geq \frac{1}{2k^2} \right\} \right) \lesssim_{n,p,\beta}  k^{2p}  2^{k\beta p}    \left|Y'_k \right|.$$ Since we have $\left| Y'_k \right| \lesssim_n 2^{-kn}\#\mathcal{I}_k $ it follows that $\mu\left(\Omega_k \right) \leq k^{2p}  2^{-k(n-\beta p)} \#\mathcal{I}_k$. We have selected $\alpha < n-\beta p$ and so we have by polynomial comparison $k^{2p}2^{-k(n-\beta p)} \lesssim_\alpha 2^{-k\alpha}$. Hence we have $$\sum_{i \in \mathcal{I}} r_i^\alpha \geq \sum_k 2^{-\alpha k}\#\mathcal{I}_k \gtrsim_\alpha \sum_k \mu(\Omega_k) \geq \mu(\Omega) = 1$$ since $\Omega$ contains the support $S_\mu$ of $\mu$.

\section{Proof of Theorem \ref{C2}}

We are going to prove Theorem \ref{C2} proving the following estimate.

\begin{thm}\label{C3}
Fix $1 < p < \infty$ and let $\mu$ be a probability on $\mathbb{S}^{n-1}$satisfying $\left[ \mu \right]_d \leq 1$ for some $0 \leq d \leq n-1$. In this case we have for any $\delta >0$ $$ \left\|K_\delta\right\|_{\mu,p} \lesssim_{n,d,p} \delta^{-\frac{n-(d+1)}{p}} \|K_\delta\|_{\sigma,p}.$$
\end{thm}

This estimate comes from the fact that \textit{a function $K_\delta f$ is almost $\delta$-discrete}. Observe that this estimate is not possible in general since the surface measure $\sigma$ typically does not charge the support $S_\mu$ of the measure $\mu$ \textit{i.e.} $\sigma(S_\mu) = 0$. The following Lemma is a manifestation of the idea that we \textit{should not} define the \textit{orientation} of an object more precisely than its \textit{eccentricity}.

\begin{lemma}\label{L1}
For any $\delta > 0$ and any directions $e_1,e_2 \in \mathbb{S}^{n-1}$ satisfying $|e_1-e_2| \leq \delta$ we have $$K_\delta f(e_1) \simeq_{n} K_\delta f(e_2) $$ for any locally integrable function $f$. 
\end{lemma}

\begin{proof}
This comes from the fact that there is a dimensional constant $a_n > 1$ such that if we have two tubes $T_{e_1,\delta},T_{e_2,\delta}$ with $|e_1-e_2| < \delta $ then one can find $\Vec{t} \in \mathbb{R}^n$ such that $$ \Vec{t} + \frac{1}{a_n}T_{e_1,\delta} \subset T_{e_2,\delta} \subset \Vec{t} + a_nT_{e_1,\delta}.$$
\end{proof}

We can then relate the $(\sigma,p)$-norm of $K_\delta f$ with a discrete sum over a family $\boldsymbol{e}_{\delta,\mathbb{S}^{n-1}} \subset \mathbb{S}^{n-1}$ which is $\delta$-separated and maximal for this property.

\begin{lemma}
For $f$ locally integrable and any family $\boldsymbol{e}_{\delta,\mathbb{S}^{n-1}}$ of $\mathbb{S}^{n-1}$ which is maximal and $\delta$-separated, we have $$\|K_\delta f \|_{\sigma,p}^p \simeq_n \sum_{e \in \boldsymbol{e}_{\delta,\mathbb{S}^{n-1}}} K_\delta f(e)^p \delta^{n-1}.$$
\end{lemma}

\begin{proof}
On one hand we have $$\int_{ \mathbb{S}^{n-1}} K_\delta f(e)^p d\sigma(e) \lesssim_{n} \sum_{e \in \boldsymbol{e}_{\delta,\mathbb{S}^{n-1}}} K_\delta f(e)^p \sigma(B(e,\delta)) \simeq_{n} \sum_{e \in \boldsymbol{e}_{\delta,\mathbb{S}^{n-1}}} K_\delta f(e)^p \delta^{n-1}.$$ On the other hand we have$$\int_{ \mathbb{S}^{n-1}} K_\delta f(e)^p d\sigma(e) \gtrsim_{n} \sum_{e \in \boldsymbol{e}_{\delta,\mathbb{S}^{n-1}}} K_\delta f(e)^p \sigma(B(e,\frac{\delta}{2})) \simeq_{n} \sum_{e \in \boldsymbol{e}_{\delta,\mathbb{S}^{n-1}}} K_\delta f(e)^p \delta^{n-1}$$ which concludes.
\end{proof}

We can now prove Theorem \ref{C3}.

\begin{proof}
Fix $\delta >0$ and consider a family $\left (e_k\right)_{k \leq m} \subset S_\mu$ which is $\delta$-separated and whose cardinal is maximal ; this implies that we have $$S_\mu \subset \bigcup_{ k \leq m } B_{e_k,2\delta}.$$ For $f$ in $L^p(\mathbb{R}^n)$ we have then $$\int_{S_\mu} K_\delta f(e)^p d\mu(e) \lesssim_{n} \sum_{k \leq m} \int_{ B_{e_k,2\delta}} K_\delta f(e_k)^p \mu( B_{e_k,2\delta}) \lesssim_{n,d} \sum_{ k \leq m } K_\delta f(e_k) \delta^d$$ using Lemma \ref{L1} and the fact that $\left[ \mu \right]_d \leq 1$. Now we complete the family $\left (e_k\right)_{k \leq m}$ into a $\delta$-separated family $\boldsymbol{e}_{\delta,\mathbb{S}^{n-1}}$ which is maximal in $\mathbb{S}^{n-1}$. We have then $$ \sum_{ k \leq m } K_\delta f(e_k)^p \delta^d \leq \sum_{ e \in \boldsymbol{e}_{\delta,\mathbb{S}^{n-1}}} K_\delta f (e)^p \delta^d \simeq \delta^{d+1-n} \sum_{ e \in \boldsymbol{e}_{\delta,\mathbb{S}^{n-1}}} K_\delta f (e)^p \delta^{n-1} \simeq_{n} \delta^{d+1-n} \|K_\delta f \|_{\sigma,p}^p$$ using the previous lemma.
\end{proof}

We can now prove Theorem \ref{C2} \textit{i.e.} we can prove that the Kakeya maximal conjecture implies to the $\Omega$-Kakeya maximal conjecture. This simply comes from the fact that the $\epsilon$-loss can be easily transferred thanks to Theorem \ref{C3}.

\begin{proof}
Fix any probability $\mu$ defined on $\mathbb{S}^{n-1}$ satisfying for some $d \in [0, n-1]$, $\left[ \mu \right]_d \leq 1$. Thanks to Theorem \ref{C3}, if the Kakeya maximal conjecture is true then we have $$\|K_\delta\|_{\mu,n} \lesssim_{n,d}\delta^{\frac{(d+1)}{n}-1}\|K_\delta\|_{ \sigma, n} \lesssim_{n,d,\epsilon} \delta^{\frac{(d+1)}{n}-(1+\epsilon)}$$ \textit{i.e.} the $\Omega$-Kakeya maximal conjecture is true.
\end{proof}

\section{Proof of Theorem \ref{THM0}}

The proof of Theorem \ref{THM0} follows Bourgain's arithmetic argument for the classic Kakeya problem ; this method relies on the following two results. The first one allow us to give an upper bound on the \textit{difference set} $A-B$.

\begin{thm}[Sum-difference Theorem]
Fix any $\delta > 0$ and suppose that $A,B$ are finite subset of $\delta \mathbb{Z}^n$ such that $\#A,\#B \leq N$. If $G \subset A\times B$ satisfies $$\#\{a+b : (a,b)\in G \} \lesssim N $$ then we have $\# \{a-b : (a,b)\in G \} \leq N^{2-\frac{1}{6}}$.
\end{thm}

The second Theorem needed is due to Heath-Brown \cite{HEATH} : it states that if $S$ is a large subset of $\{0,\dots,M\}$ for $M$ large enough then $S$ contains an arithmetic progression of length $3$.

\begin{thm}[Heath-Brown]
There exists an integer $M_0$ such that if $M > M_0$ is a integer and if $S$ is a subset of $\{0,\dots,M \}$ such that $$\#S \geq \frac{M}{\log(M)^c} $$ then $S$ contains a subset of the form $\{m,m+m', m+2m' \} \subset S$. Here $c>0$ is an absolute constant.
\end{thm}

For the sake of clarity, we have decomposed the proof of Theorem \ref{THM0} in two steps. We will denote by $C(A,\delta)$  the smallest number of balls of radius $\delta$ needed to cover the set $A$.

\subsection*{Decomposition of the $\Omega$-Kakeya set}

To begin with, we may suppose that $\Omega$ is contained in a small spherical cap ; concretely we suppose that for any $e = (e_1,\dots,e_n) \in \Omega$ we have $e_n > \frac{1}{2}$. We fix then $d < d_\Omega$ arbitrarily close and we use Frostman's Lemma to obtain an probability $\mu$ such that $S_\mu \subset \Omega$ and also $[\mu]_d \leq 1$. We consider then a $\Omega$-Kakeya set $X$ and we suppose that $X$ is contained in $[0,1]^n$. For any $e \in \Omega$ we will denote by $T_e$ a unit segment oriented along $e$ contained in $X$. Finally we fix $s > d_X$ and we will prove that we have $$s \geq \frac{6}{11}d +1.$$

We fix $\epsilon\in (0,1)$ arbitrarily small and we let $ \eta \in (0,1)$ such that defining $\delta_k = 2^{-2^{\eta k}}$ we have for any $k \geq 1$, $$ \delta_{k}^{s+\epsilon} \leq \delta_{k+1}^s.$$ Since $s > d_X$, for arbitrary large $k_0$, we can cover $X$ by a countable collection of balls $\{ B_i \}_{i \in \mathcal{I}}$ and such that for any $i \in \mathcal{I}$, we have $\text{diam}( B_i )  < \delta_{k_0}$ and $ \sum_{i \in \mathcal{I}} \text{diam}( B_i)^s < 1$. In addition, we take $k_0$ so large that we have $$k_0^4  \max( \delta_{k_0}^\eta ,\delta_{k_0}^{\eta d})  < 1.$$ We denote by $Y$ the union of the balls $\{B_i \}_{i \in \mathcal{I}}$ \textit{i.e.} $$Y := \bigcup_{i \in \mathcal{I}} B_i $$ and for $k \geq k_0$ we will denote by $ \mathcal{I}_k := \{ i \in \mathcal{I}_k : \delta < \text{diam}(B_i) \leq \delta \}$ and also $Y_k := \bigcup_{ i \in \mathcal{I}_k } B_i$. We can control the size of $\#\mathcal{I}_k$.

\begin{claim}
We have $ \#\mathcal{I}_k \delta_k^{s+\epsilon} < 1$.
\end{claim}

\begin{proof}
By definition of $\mathcal{I}_k$ and since we have $\delta_{k}^{s+\epsilon} \leq \delta_{k+1}^s$ and $\sum_{i \in \mathcal{I}}  \text{diam}( B_i)^s < 1 $, we obtain $$  \#\mathcal{I}_k \delta_k^{s+\epsilon} \leq  \#\mathcal{I}_k \delta_{k+1}^s  \leq \sum_{i \in \mathcal{I}_k} \text{diam}(B_i)^s < 1.$$
\end{proof}

\subsubsection*{step 1 : refinement to a single scale}

We wish to work at a single scale with respect to this covering. Hence we are going to exhibit a $k \geq k_0$ such that there is a specific subset $\Omega_k \subset \Omega$ adapted to the covering $\{ B_i \}_{i \in \mathcal{I}_k}$ : on one hand $\Omega_k$ is large and on the other hand, for any $e \in \Omega_k$, the unit segment $T_e \subset X$ is well covered by the balls $\{B_i\}_{i \in \mathcal{I}_k}$.

\begin{claim}
There exists $k \geq k_0$ and $\Omega_k \subset \Omega$ such that for any $e \in \Omega_k$, $$|T_e \cap Y_k|_1 \geq \frac{1}{k^2}$$ and also $\mu(\Omega_k) > \frac{1}{k^2}$.
\end{claim}

\begin{proof}
If this is not the case, then for any $k \geq k_0$ we have $$\mu \left( \{ e \in \Omega : |T_e \cap Y_k|_1 \geq \frac{1}{k^2} \} \right) \leq \frac{1}{k^2}.$$ Hence we have $$\mu \left( \{ e\in \Omega : \exists k \geq k_0, |T_e \cap Y_k|_1 \geq \frac{1}{k^2}  \} \right) \leq \sum_{k \geq k_0} \frac{1}{k^2} < \mu( \Omega) $$ and so there is $e \in \Omega$ such that $|T_e \cap Y_k|_1 < \frac{1}{k^2}$ for any $k \geq k_0$. This is not possible since $Y$ covers $X$ and so $T_e$ in particular.
\end{proof}

\subsubsection*{step 2 : slicing $\mathbb{R}^n$ at two scales}

We fix such a $k$ and we let $\delta := \delta_k$. Recall that since $k \geq k_0$ and that we can choose $k_0$ arbitrarily large, the same is true for $k$ \textit{i.e.} the integer $k$ can be chosen arbitrarily large. Also observe that by definition we have $$k \simeq \log\log( \delta^{-\eta}).$$ Now we fix two integers $N,M \in \mathbb{N}$ such that $$(N,M) \simeq (\delta^{\eta-1},\delta^{-\eta}).$$ We are going to slice $\mathbb{R}^n$ at two different scales ($\delta$ and $\delta^\eta$) along the vector $(0,\dots,1)$. Precisely for $j \leq N$ and $m \leq M$, we define  $$A_{j,m} := \{ x=(x_1,\dots,x_n) \in \mathbb{R}^n : j\delta + mN\delta \leq x_n \leq (j+1)\delta + mN\delta \}$$ and $A_j := \bigcup_{m \leq M} A_{j,m}$.

\begin{claim}
For any $e \in \Omega_k$ and $j \leq N$, we have $|T_e \cap A_j|_1 \simeq M\delta \simeq \frac{1}{N}$.
\end{claim}

\begin{proof}
The claim comes from the fact that we have supposed that for any $e=(e_1,\dots,e_n) \in \Omega$ we have $e_n > \frac{1}{2}$.
\end{proof}

By definition of $\Omega_k$, we also have the following estimate $$ \frac{1}{k^4} \leq \frac{\mu(\Omega_k)}{k^2}  \leq \int_{\Omega_k}  | Y_k \cap T_e |_1 d\mu(e)  = \sum_{ j \leq N } \int_{\Omega_k} |Y_k \cap T_e \cap A_j |_1 d\mu(e). $$ We define then the subset $J \subset \{0,\dots,N\}$ as $$ J = \{ j \leq N :  \int_{\Omega_k} |Y_k \cap T_e \cap A_j |_1 \geq \frac{1}{2Nk^4}  \}.$$ The following claim states that this set $J$ is not too small in $\{0,\dots,N\}$.

\begin{claim}
We have $ \#J \gtrsim \delta^{\eta} N$
\end{claim}

\begin{proof}
The proof comes from a reverse Markov inequality.
\end{proof}

Now for each $j \in J$, we extract a subset $\Omega_{k,j}$ from $\Omega_k$ in the same fashion that we have extracted $\Omega_k$ from $\Omega$. The proof is the same than for Claim 2.

\begin{claim}
For any $j \in J$, there exists $\Omega_{k,j} \subset \Omega_k$ such that for any $e \in \Omega_{k,j}$, $$|T_e \cap Y_k \cap A_j|_1 > \frac{|T_e \cap A_j|_1}{4k^4}$$ and also $\mu(\Omega_{k,j}) > \frac{ \mu(\Omega_k) }{ k^2}$.
\end{claim}

\subsubsection*{step 3 : conclusion}

Observe that for $|j-j'| > 2$, the sets $Y_k \cap A_j$ and $Y_k \cap A_{j'}$ are separated by a distance greater than $2\delta$. Hence, on one hand we have $$\sum_{j \in J} C(Y_k \cap A_j,\delta) \lesssim C(Y_k,\delta) \lesssim \delta^{-s-\epsilon}.$$  On the other hand, suppose that for any $j \in J$ we have $$C(Y_k \cap A_j,\delta) \gtrsim \delta^{ \frac{6}{11}(2\eta-1)d }.$$ In this case we obtain $$\sum_{j \in J} C(Y_k \cap A_j,\delta) \gtrsim \#J \times \delta^{ \frac{6}{11}(2\eta-1)d} \gtrsim \delta^{  \frac{6}{11}(2\eta-1)d + (2\eta-1)}.$$ Since $\delta$ is small enough, we get  $\frac{6}{11}(1 -2\eta)d  - 2\eta + 1 \leq   s + \epsilon$. Taking $\epsilon$ and $\eta$ arbitrarily small we conclude that $$s \geq \frac{6}{11}d +1.$$

\subsection*{Lower bound for $C(Y_k \cap A_j,\delta)$}

Hence we are left to prove that we have, for any $j \in J$, the following bound $$C(Y_k \cap A_j,\delta) \gtrsim \delta^{ \frac{6}{11}(2\eta-1)d }.$$ We will start by applying Heath-Brown's Theorem and we will use thereafter the Sum-difference Theorem to obtain a bound on $C(Y_k \cap A_j,\delta)$.

\subsubsection*{step 1 : application of Heath-Brown's Theorem}

For any $j \in J$ and any $e \in \Omega_{k,j}$, we consider the following subset of $\{0,\dots,M\}$ $$K(e,j) := \{ m \leq M : Y_k \cap T_e \cap A_{j,m} \neq \emptyset \}.$$ The following claim states that $K(e,j)$ contains a lot of element in $\{0,\dots,M\}$.

\begin{claim}
We have $\#K(e,j) \gtrsim \frac{M}{\log\log(M)}$.
\end{claim}

\begin{proof}
We have $$\#K(e,j) \times \delta \gtrsim | Y_k \cap T_e \cap A_j |_1 > \frac{1}{4Nk^4}.$$ Hence $$ \#K(e,j) \gtrsim \frac{M}{k^4} \simeq \frac{M}{\log \log(M)}.$$ 
\end{proof}

Since we can take $M$ arbitrary large and that $K(e,j)$ is quite large in $\{0,\dots,M \}$, we are able to exhibit arithmetic progressions of three terms in $K(e,j)$ using the Theorem of Heath-Brown.

\begin{claim}
For any $j \in J$ and $e \in \Omega_{k,j}$, the set $K(e,j)$ contains a subset of the form $$\{ m, m+m' , m+2m' \} \subset K(e,j).$$ 
\end{claim}

Hence for any $j \in J$ and any $e \in \Omega_{k,j}$, there exists $a_e, b_e \in Y_k(n\delta) \cap T_e(n\delta) \cap A_j \cap \delta \mathbb{Z}^n$ such that $$\frac{a_e+b_e}{2} \in Y_k(n\delta) \cap T_e(n\delta) \cap A_j \cap \delta \mathbb{Z}^n$$ and $a_e,b_e$ belong to different sets $A_{j,m}$. Observe also that the sets $A_{j,m}$ for different indices $m$ are at least distant of $\simeq \delta^\eta$. Thus if $\delta$ is small enough we have $|a_e-b_e| \gtrsim \delta^\eta$. Finally, we consider the sets $A = \{ a_e : e\in \Omega_{k,j} \}$, $B = \{ b_e : e \in \Omega_{k,j} \}$ and $$G := \{ (a_e,b_e) : e\in \Omega_{k,j} \} \subset A\times B.$$ Recall that we have $$A,B \subset \delta\mathbb{Z}^n.$$

\subsubsection*{step 2 : upper bound for $\#\{ a-b :(a,b) \in G \}$}

We are going to give an upper bound and lower bound on $\#\{ a-b :(a,b) \in G \}$. Observe that the cardinal of $A,B$ and $\{ a+b :(a,b) \in G \}$ is controlled by the covering number $C(Y_k \cap A_j,\delta)$. Hence a direct application of the Sum-difference Theorem yields an upper bound on $\#\{ a-b :(a,b) \in G \}$

\begin{claim}
We have $\#\{ a-b :(a,b) \in G \}  \lesssim  C(Y_k \cap A_j,\delta)^{ \frac{11}{6} }$.
\end{claim}

\subsubsection*{step 3 : lower bound for $\#\{ a-b :(a,b) \in G \}$}

Finally we are also able to provide a lower bound on $\#\{ a-b :(a,b) \in G \}$ using the measure $\mu$.

\begin{claim}
We have $\delta^{(2\eta-1)d}  \lesssim \#\{ a-b :(a,b) \in G \}$.
\end{claim}

\begin{proof}
Since $a_e$ and $b_e$ are in the $n\delta$-neighbourhood of $T_e$ and $|a_e  - b_e| > \frac{\delta^\eta}{2}$ for $e \in \Omega_{k,j}$, it follows that balls roughly of radius $\delta^{1-\eta}$ centred at the unit vectors $\frac{a_e-  b_e}{|a_e - b_e|}$ (for $e \in \Omega_{k,j}$) cover $\Omega_{k,j}$. As $\mu(\Omega_{k,j})> \frac{1}{4k^4}$,  this implies $$\#(A-B) \gtrsim \frac{\delta^{(\eta-1)d}}{k^4} \gtrsim \delta^{(2\eta-1)d }.$$
\end{proof}

Hence for any $j \in J$, we have $$ C(Y_k \cap A_j,\delta) \gtrsim \delta^{ \frac{6}{11}(2\eta-1)d }.$$ This concludes the proof of Theorem \ref{THM0}.

{}

\end{document}